\documentclass[11pt]{article}

\pdfoutput=1

\makeatletter
\let\@fnsymbol\@arabic
\makeatother

\makeatletter
\def\blfootnote{\gdef\@thefnmark{}\@footnotetext}
\makeatother

\title{Characterization of geodesic completeness for landmark space\blfootnote{{\it $2020$ Mathematics Subject Classification.} 53C22, 70H05, 58J65, 53Z50}}
\author{Karen Habermann\,\footnote{Department of Statistics, University of Warwick, Coventry, CV4 7AL, United Kingdom.\\ {\it Email address:} {\tt karen.habermann@warwick.ac.uk}} , Stephen C. Preston\,\footnote{Department of Mathematics, CUNY Brooklyn College, New York, NY 11210-2889 and CUNY
Graduate Center, New York, NY 10016, USA. {\it Email address:} {\tt Stephen.Preston@brooklyn.cuny.edu}} , Stefan Sommer\,\footnote{Department of Computer Science, University of Copenhagen, Universitetsparken 5, DK-2100 Copenhagen E, Denmark. {\it Email address:} {\tt sommer@di.ku.dk}}}

\usepackage{amsmath, amssymb, amsthm, stmaryrd, fullpage, graphicx, color,xcolor,hyperref}

\hypersetup{
    pdftitle={Characterization of geodesic completeness for landmark space},
    pdfauthor={Karen Habermann, Stephen C. Preston, Stefan Sommer},
}

\newcommand{\kernl}{\mathcal{K}}

\newcommand{\R}{\mathbb{R}}
\newcommand{\Diff}{\mathrm{Diff}}
\newcommand{\dd}{\,{\mathrm d}}
\newcommand{\db}{{\mathrm d}}

\newcommand{\lm}{\textbf{x}}
\newcommand{\mom}{\textbf{p}}
\newcommand{\ham}{\mathcal{H}}

\newtheorem{theorem}{Theorem}

\newtheorem{lemma}[theorem]{Lemma}
\newtheorem{proposition}[theorem]{Proposition}
\theoremstyle{definition}

\theoremstyle{remark}
\newtheorem{remark}[theorem]{Remark}
\newtheorem{example}[theorem]{Example}

\begin{document}
\maketitle

\begin{abstract}
    We provide a full characterization of geodesic completeness for spaces of configurations of landmarks with smooth Riemannian metrics that satisfy a rotational and translation invariance and which are induced from metrics on subgroups of the diffeomorphism group for the shape domain. These spaces are widely used for applications in shape analysis, for example, for measuring shape changes in medical imaging and morphometrics in biology. For statistics of such data to be well-defined, it is imperative to know if geodesics exist for all times. We extend previously known sufficient conditions for geodesic completeness based on the regularity of the metric to give a full characterization for smooth Riemannian metrics with a rotational and translation invariance by means of an integrability criterion that involves only the behavior of the cometric kernel as landmarks approach collision. We further use the integrability criterion for geodesic completeness and previous work on stochastic completeness to construct a family of Riemannian landmark manifolds that are geodesically complete but stochastically incomplete.
\end{abstract}

\section{Introduction}
In shape analysis, particularly the large deformation diffeomorphic mapping framework (LDDMM), see, for example,~\cite{trouveDiffeomorphismsGroupsPattern1998,younesComputableElasticDistances1998,younesShapesDiffeomorphisms2010}, subgroups of the diffeomorphism group of a shape domain are equipped with right-invariant metrics, typically arising from Sobolev operators. Diffeomorphisms act on a large number of shape spaces, and, due to the right-invariance, such metrics descend to orbit shape spaces. One of the most fundamental of such cases is when shapes are configurations of finite sets of ordered distinct landmarks in $\R^d$, denoted by $\lm=(x_1,\dots,x_n)$ where $x_i\in\R^d$ and $x_i\not= x_j$ for $i,j\in\{1,\dots, n\}$ with $i\not= j$. The action of a suitable diffeomorphism $\phi\in\Diff(\R^d)$ is then by composition from the left, that is, $\phi.\lm=(\phi(x_1),\dots,\phi(x_n))$, and the induced metric gives the landmark configuration space the structure of a finite-dimensional Riemannian manifold, as introduced in~\cite{joshiLandmarkMatchingLarge2000}. These spaces have been the subject of intensive interest in the literature, both mathematically as in~\cite{micheliDifferentialGeometryLandmark2008,MMM} and for applications in, among others, medical imaging and morphometrics in biology, see~\cite{younesEvolutionsEquationsComputational2009}. They are frequently used for statistics of shapes when landmarks represent points of interest on datasets of shapes, extending shape statistics as seen in Kendall's shape space beyond the Euclidean setting. Such statistics are often based on geodesics between point configurations. Therefore, both for theoretical interest and for applications, it is important to know if Riemannian landmark manifolds are geodesically complete.

For $n\geq 2$ and $d\geq 1$ fixed, let $(M,g)$ denote the landmark configuration space consisting of $n$ distinct landmarks in $\R^d$. In particular, the manifold $M$ is given by
\begin{displaymath}
    M=\{\lm=(x_1,\dots,x_n):x_1,\dots,x_n\in\R^d\text{ with } x_i\not=x_j\text{ for }i\not=j\}\subset\R^{nd}.
\end{displaymath}
When the Riemannian metric $g$ is inherited from a right-invariant weak Riemannian metric on a suitable subgroup of the diffeomorphism group $\Diff(\R^d)$, it can be described uniquely by its cometric that takes the form
\begin{displaymath}
    g^{ij}(\lm)=K(x_i,x_j),
\end{displaymath}
for $i,j\in\{1,\dots,n\}$ and a positive-definite symmetric kernel $K\colon\R^d\times\R^d\to\R^{d\times d}$. Throughout the article, we assume that the cometric kernel $K$ is rotationally and translationally invariant so that there exists a continuous scalar function $\kernl\colon [0,\infty)\to\R$ such that
\begin{displaymath}
    K(x_i,x_j)=\kernl(\|x_i-x_j\|) I_d,
\end{displaymath}
where $\|\cdot\|$ is the Euclidean norm on $\R^d$ and $I_d$ is the $d\times d$ identity matrix. In particular, we then have, for all $i,j\in\{1,\dots,n\}$,
\begin{equation}\label{eq:defnginv}
    g^{ij}(\lm)=\kernl(\|x_i-x_j\|) I_d.
\end{equation}
As a consequence of Schoenberg~\cite[Theorem~3]{schoenberg}, the function $\kernl\colon [0,\infty)\to\R$ is guaranteed to be positive and strictly decreasing.

\medskip

Geodesics on the Riemannian manifold $(M,g)$ can be studied by lifting to a suitable subgroup of $\Diff(\R^d)$. As discussed by Bauer, Bruveris and Michor in~\cite{bauerOverviewGeometriesShape2014}, since geodesics of diffeomorphisms are well-defined for all times if $K\in C^1(\R^d\times\R^d,\R^{d\times d})$ and because geodesics on the landmark space arise as projections of certain geodesics of diffeomorphisms, we have the following sufficient condition for geodesic completeness.
\begin{theorem}[Bauer--Bruveris--Michor~\cite{bauerOverviewGeometriesShape2014}]\label{thm:BBM}
    Let $(M,g)$ be the landmark space for any number of landmarks in any $\R^d$. If $K\in C^1(\R^d\times\R^d,\R^{d\times d})$ then $(M,g)$ is geodesically complete. Furthermore, if $K\in C^2(\R^d\times\R^d,\R^{d\times d})$ then $(M,g)$ is metrically complete.
\end{theorem}
The second part of Theorem~\ref{thm:BBM} follows from an application of the Hopf--Rinow theorem, which in the $C^2$ case gives the equivalence between geodesic completeness and metric completeness. While the above approach covers a large class of metrics, geodesic completeness results for $(M,g)$ with metrics corresponding to non-$C^1$ cometric kernels have so far not been established.

\medskip

In this article, we fully characterize geodesic completeness for landmark spaces subject to the assumed rotational and translational invariance in terms of a single integrability criterion for the scalar function $\kernl\colon [0,\infty)\to\R$ close to zero, that is, close to landmarks colliding. Specifically, we prove the following theorem that does not assume $K\in C^1(\R^d\times\R^d,\R^{d\times d})$.
\begin{theorem}\label{thm:main}
    Let $(M,g)$ be the landmark space for any number of landmarks in any $\R^d$. Suppose that $\kernl\colon[0,\infty)\to\R$ is smooth on $(0,\infty)$. Then $(M,g)$ is geodesically complete if and only if, for all $a>0$,
    \begin{displaymath}
        \int_0^a\frac{\db r}{\sqrt{\kernl(0)-\kernl(r)}}=\infty.
    \end{displaymath}
\end{theorem}

We prove Theorem~\ref{thm:main} in two parts by establishing the following two propositions.
\begin{proposition}[Geodesic incompleteness]\label{propn:geoincomplete}
    Let $(M,g)$ be the landmark space for any number of landmarks in any $\R^d$. Suppose that $\kernl\colon[0,\infty)\to\R$ is smooth on $(0,\infty)$ and satisfies, for some $a>0$,
    \begin{equation}\label{eq:kernlintegrabl}
        \int_0^a\frac{\db r}{\sqrt{\kernl(0)-\kernl(r)}}<\infty.
    \end{equation}
   Then $(M,g)$ is geodesically incomplete.
\end{proposition}
We prove Proposition~\ref{propn:geoincomplete} by explicitly identifying geodesic paths which cannot be extended to all times, in the sense that two landmarks collide in finite time. We further observe that the integrability criterion~\eqref{eq:kernlintegrabl} is true for all $a>0$ if it is satisfied for some $a>0$.

\begin{proposition}[Geodesic completeness]\label{propn:geocomplete}
    Let $(M,g)$ be the landmark space for any number of landmarks in any $\R^d$. Suppose that $\kernl\colon[0,\infty)\to\R$ is smooth on $(0,\infty)$ and satisfies, for all $a>0$,
    \begin{displaymath}
        \int_0^a\frac{\db r}{\sqrt{\kernl(0)-\kernl(r)}}=\infty,
    \end{displaymath}
    that is, the improper integral does not converge. Then $(M,g)$ is geodesically complete.
\end{proposition}
We prove Proposition~\ref{propn:geocomplete} by first showing that every curve which tends to the collision set
\begin{displaymath}
    C=
    \R^{nd}\setminus M
    =\{\lm=(x_1,\dots,x_n):x_1,\dots,x_n\in\R^d\text{ with } x_i=x_j\text{ for some }i\not=j\}
\end{displaymath}
or to Euclidean infinity in finite time must have infinite length, and then applying the Hopf--Rinow theorem. It is worthwhile to remark that no additional criterion is needed to exclude tending to Euclidean infinity in finite time.

Figure~\ref{fig:kernels} illustrates the distinctive behavior near zero of the two kernels which correspond to $\kernl(r)=\exp(-r)$ and $\kernl(r)=2(1+r)\exp(-r)$, respectively, and which give rise to geodesically incomplete and geodesically complete landmark spaces, respectively. Note that the latter is in line with Theorem~\ref{thm:BBM} because $(x_i,x_j)\mapsto 2(1+\|x_i-x_j\|)\exp(-\|x_i-x_j\|)$ is $C^1$ in both components.
\begin{figure}[ht]
    \centering
    \includegraphics[width=0.7\textwidth]{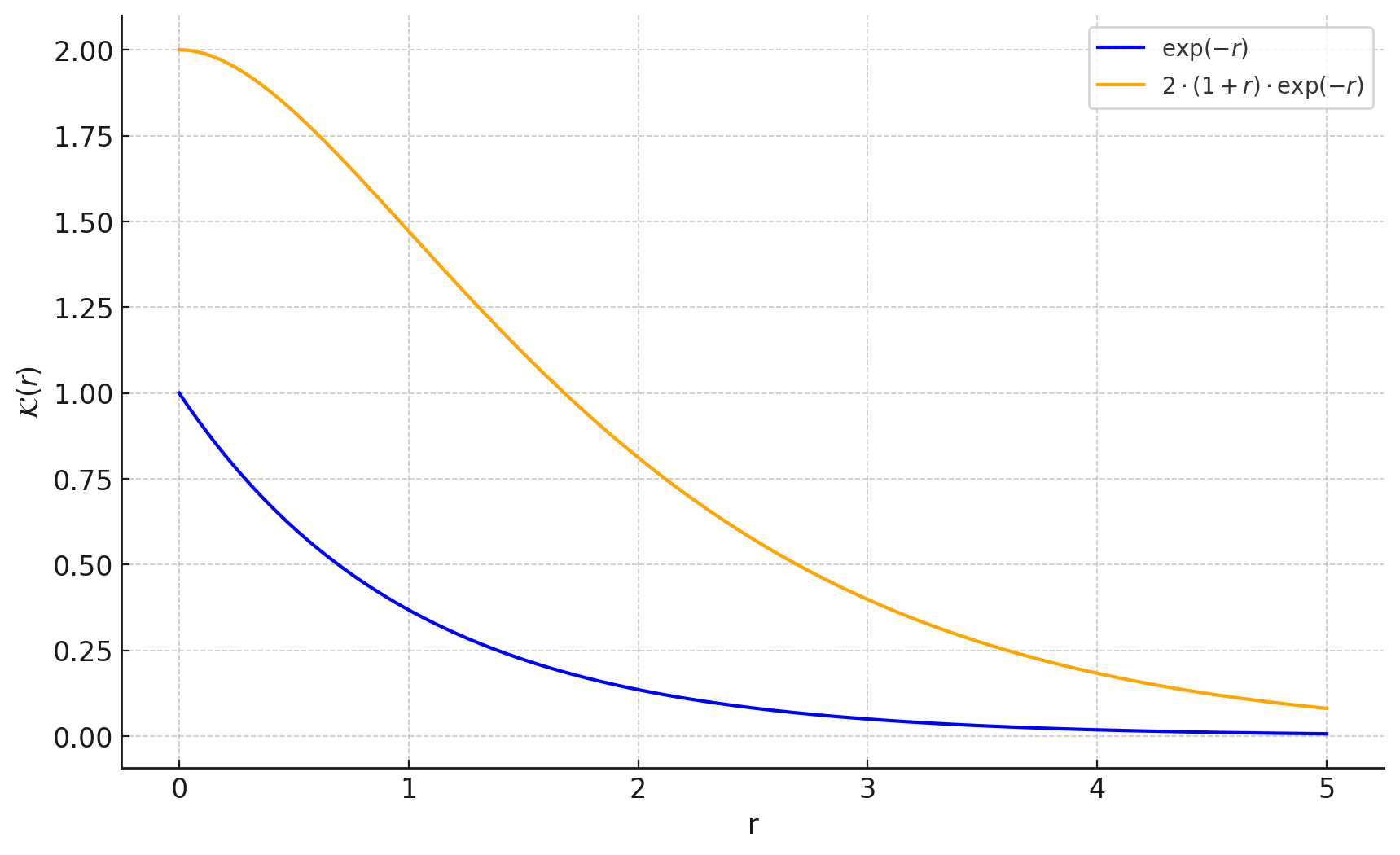}
    \caption{Illustrating kernels given by $\mathcal K(r) = \exp(-r)$ and $\mathcal K(r) = 2 (1 + r) \exp(-r)$, respectively, where $\mathcal K$ represents the cometric kernel of a landmark configuration space. The integrability condition~\eqref{eq:kernlintegrabl} is not satisfied for the former leading to geodesic incompleteness, while the latter satisfies the integrability condition and thus gives geodesic completeness.}
    \label{fig:kernels}
\end{figure}

\medskip

If the scalar function $\kernl\colon[0,\infty)\to\R$ satisfies $\kernl(0)-\kernl(r)=D r^\gamma+o(r^\gamma)$ for $D,\gamma>0$ as $r\downarrow 0$, then, by Theorem~\ref{thm:main}, the associated landmark spaces are geodesically complete if and only if $\gamma\geq 2$. On landmark spaces of exactly two landmarks, this threshold coincides with the one for stochastic completeness derived in~\cite{LongtimeExistenceBrownian2024}. However, by considering more general kernels, we obtain landmark spaces that are geodesically complete but stochastically incomplete. This is of interest in its own right because it provides new examples of Riemannian manifolds where the two notions of completeness do not coincide. The following is discussed in more detail in Section~\ref{sec:geostodiff}.
\begin{theorem}[Geodesically complete but stochastically incomplete landmark space]\label{thm:geostodiff}
    Fix $d=2$ and some constant $c\in(1,2]$. Let $\kernl\colon[0,\infty)\to\R$ be smooth on $(0,\infty)$ and such that, for $r\in[0,\frac{1}{2}]$,
    \begin{displaymath}
        \kernl(r)=1-r^2\left(1-\log(r)\right)^c.
    \end{displaymath}
    Then the associated landmark space $(M,g)$ of exactly two landmarks is geodesically complete but stochastically incomplete.
\end{theorem}
We will see that Theorem~\ref{thm:geostodiff} is a consequence of having chosen the function $\kernl\colon[0,\infty)\to\R$ such that, for $a>0$,
\begin{displaymath}
    \int_0^a\frac{\dd r}{\sqrt{\kernl(0)-\kernl(r)}}=\infty
    \quad\text{and}\quad
    \int_0^a\frac{r}{\kernl(0)-\kernl(r)}\dd r<\infty,
\end{displaymath}
which cannot be achieved if we have $\kernl(0)-\kernl(r)=D r^\gamma+o(r^\gamma)$ for $D,\gamma>0$ as $r\downarrow 0$.

\medskip

By extending the analysis in the Hamiltonian approach used to prove Proposition~\ref{propn:geoincomplete}, we will further demonstrate that in the two-landmark system only very special choices of initial conditions, namely those where the total angular momentum is zero, lead to geodesics that cannot be extended to all times. Specifically, we establish the following theorem, which is stated more precisely as Theorem~\ref{thm:twolandmarksrig} in Section~\ref{sec:twolandmarksystem}.
\begin{theorem}\label{thm:twolandmarks}
    For the landmark space of exactly two landmarks in any $\R^d$, geodesic paths with non-zero total angular momentum can be extended to all times, regardless of the choice of smooth kernel $\kernl\colon[0,\infty)\to\R$.
\end{theorem}

The behavior observed for exactly two landmarks is comparable with other two-body problems where, for example, a body can fall toward the sun leading to a finite-time collapse, but in the presence of non-zero angular momentum, it will orbit around the sun forever.

\subsection{Paper outline}
In Section~\ref{sec:hamiltonian_dynamics}, we revisit the Hamiltonian formulation of the geodesic equation on landmark spaces and we prove conservation of the total linear momentum as well as of the total angular momentum. We use this in Section~\ref{sec:geoincomplete} to prove Proposition~\ref{propn:geoincomplete}, that is, the part of Theorem~\ref{thm:main} which guarantees geodesic incompleteness. We proceed by proving Proposition~\ref{propn:geocomplete}, that is, the part of Theorem~\ref{thm:main} which ensures geodesic completeness, in Section~\ref{sec:geocomplete}. We continue by demonstrating in Section~\ref{sec:geostodiff} that Theorem~\ref{thm:geostodiff} indeed provides examples of landmark spaces that are geodesically complete but stochastically incomplete. We then close by analyzing the two-landmark system in more detail in Section~\ref{sec:twolandmarksystem} and, in particular, by establishing Theorem~\ref{thm:twolandmarks} that the presence of non-zero total angular momentum in this case prevents finite-time breakdown.

\paragraph{Acknowledgements.}
The authors thank the organizers and participants of the Summer Shape Workshops (\href{https://sites.google.com/view/shape-workshop/home}{https://sites.google.com/view/shape-workshop/home}) ``Math in the Mine'' in 2023 and ``Math in Maine'' in 2024 , where most of this work was done.   The third author is supported by a research grant (VIL40582) from VILLUM FONDEN and the Novo Nordisk Foundation grants NNF18OC0052000, NNF24OC0093490, NNF24OC0089608.

\section{Hamiltonian dynamics on landmark space}
\label{sec:hamiltonian_dynamics}
For the landmark manifold $(M,g)$, the easiest way to derive the geodesic equation is to take the Hamiltonian approach on the cotangent bundle $T^\star M$. This is due to the Hamiltonian approach working in terms of the cometric of the metric $g$, whose components explicitly defined by~\eqref{eq:defnginv} are relatively simple, rather than the metric $g$ itself, whose components become quite cumbersome for more than two landmarks.

We use coordinates $(\lm,\mom)=(x_1,\dots,x_n,p_1,\dots,p_n)$ for the cotangent bundle $T^\star M$ with $p_i\in\R^d$ for all $i\in\{1,\dots,n\}$. The Hamiltonian $\ham\colon T^\star M\to\R$ for the landmark space $(M,g)$ is then given by, in terms of the Euclidean inner product and the Euclidean norm on $\R^d$,
\begin{equation}\label{eq:hamiltonian}
    \ham=\frac{1}{2}\sum_{i,j=1}^n\langle p_i,g^{ij}(\lm)p_j \rangle
    =\frac{1}{2}\sum_{i,j=1}^n\kernl(\|x_i-x_j\|)\langle p_i,p_j\rangle.
\end{equation}
The associated Hamilton's equations take the form, for $i\in\{1,\dots,n\}$,
\begin{align}
    \frac{\db x_i}{\db t}\label{eq:hamx}
    &=\sum_{j=1}^n \kernl(\|x_i-x_j\|)p_j,\\
    \frac{\db p_i}{\db t}\label{eq:hamp}
    &=-\sum_{\substack{j=1\\j\not=i}}^n \kernl'(\|x_i-x_j\|)\frac{x_i-x_j}{\|x_i-x_j\|}\langle p_i,p_j\rangle.
\end{align}
Here, the notation used means that these equations for vector quantities $x_i,p_i\in\R^d$ are satisfied separately for each component in $\R^d$. We note that the expression on the right-hand side of Hamilton's equation~\eqref{eq:hamp} is well-defined on $M$ because $x_i\not=x_j$ for $i\not= j$.

The system of Hamilton's equations~\eqref{eq:hamx} and \eqref{eq:hamp} consists of coupled first-order equations which describe the geodesic equation on the landmark space $(M,g)$ in such a way that geodesics on the manifold $M$ arise as projections of solutions in $T^\star M$ to Hamilton's equations.

In our subsequent analysis, we crucially use that for any solution to the system of Hamilton's equations~\eqref{eq:hamx} and \eqref{eq:hamp} the total momentum is conserved. It is a consequence of the Hamiltonian $\ham$ being invariant under translation symmetry along with Noether's theorem.
\begin{lemma}[Conservation of momentum]\label{lem:conservationofmomentum}
    Let $(\lm,\mom)\colon [0,T)\to T^\star M$ for $T>0$ be a solution to the system of Hamilton's equations~\eqref{eq:hamx} and \eqref{eq:hamp}. Then the total momentum $P\colon [0,T)\to\R^d$ defined by, for $t\in[0,T)$,
    \begin{equation}\label{eq:totalmom}
        P(t)=\sum_{i=1}^n p_i(t)
    \end{equation}
    is conserved.
\end{lemma}
\begin{proof}
    It follows from~\eqref{eq:hamp}, $\langle p_i,p_j\rangle=\langle p_j,p_i\rangle$ and $x_j-x_i=-(x_i-x_j)$ that
    \begin{displaymath}
        \frac{\db P}{\db t}
        =\sum_{i=1}^n\frac{\db p_i}{\db t}
        =-\sum_{i=1}^n\sum_{\substack{j=1\\j\not=i}}^n \kernl'(\|x_i-x_j\|)\frac{x_i-x_j}{\|x_i-x_j\|}\langle p_i,p_j\rangle=0,
    \end{displaymath}
    which implies the claimed result.
\end{proof}

In addition to the conservation of the total linear momentum, the total angular momentum of the considered Hamiltonian system is also conserved as a consequence of rotational invariance. The latter can be expressed conveniently in terms of the wedge product which, for two vectors $y,z\in \R^d$, is the $\frac{1}{2}d(d-1)$-dimensional vector $y\wedge z$ whose entries are $y_k z_\ell - y_\ell z_k$ for $k,\ell\in\{1,\dots,d\}$ with $k<\ell$. Notably, we have $y\wedge z=-z\wedge y$.
\begin{lemma}[Conservation of angular momentum]\label{lem:conservationofangmom}
    For a solution $(\lm,\mom)\colon [0,T)\to T^\star M$ with $T>0$ to the system of Hamilton's equations~\eqref{eq:hamx} and \eqref{eq:hamp}, the total angular momentum $L\colon [0,T)\to \R^{\frac{d(d-1)}{2}}$ given by, for $t\in[0,T)$,
    \begin{displaymath}
        L(t)=\sum_{i=1}^n x_i(t)\wedge p_i(t)
    \end{displaymath}
    is conserved.
\end{lemma}
\begin{proof}
    Applying Hamilton's equations~\eqref{eq:hamx} and \eqref{eq:hamp} as well as using the antisymmetry of the wedge product, we obtain
    \begin{displaymath}
        \frac{\db L}{\db t}
        = \sum_{i,j=1}^n \kernl(\|x_i-x_j\|)p_j\wedge p_i
        -\sum_{i=1}^n\sum_{\substack{j=1\\j\not=i}}^n \kernl'(\|x_i-x_j\|)\frac{x_i\wedge(x_i-x_j)}{\|x_i-x_j\|}\langle p_i,p_j\rangle=0,
    \end{displaymath}
    and the claimed result follows.
\end{proof}

We remark that Lemma~\ref{lem:conservationofangmom} can alternatively be proven by showing that the Poisson bracket $\{L,H\}$ is constantly equal to zero. The conservation of the total angular momentum will be exploited in the analysis of the two-landmark system performed in Section~\ref{sec:twolandmarksystem}.

\section{Geodesically incomplete landmark space}
\label{sec:geoincomplete}

We establish Proposition~\ref{propn:geoincomplete} by means of geodesic shooting, where we explicitly identify initial conditions for which solutions to Hamilton's equations cannot be extended to all times, in the sense that two landmarks collide. For this analysis, we restrict our attention to initial conditions for which Hamilton's equations simplify sufficiently so that the corresponding solutions can be understood well enough.

\begin{lemma}\label{lem:zeromomenta}
    Suppose that $(\lm,\mom)\colon [0,T)\to T^\star M$ for $T>0$ is a solution to the system of Hamilton's equations~\eqref{eq:hamx} and \eqref{eq:hamp} with $p_i(0)=0$ for some $i\in\{1,\dots,n\}$. Then $p_i(t)=0$ for all $t\in[0,T)$.
\end{lemma}
\begin{proof}
    We see from~\eqref{eq:hamp} that each $p_i$ for $i\in\{1,\dots,n\}$ satisfies a differential equation of the form
    \begin{displaymath}
        \frac{\db p_i}{\db t}=\langle p_i(t),\Phi_i(t)\rangle,
    \end{displaymath}
    for $\Phi_i$ some vector-valued function of time. In the full system, $\Phi_i(t)$ depends on the various $x_j(t)$ as well as $p_j(t)$. However, once we solve this system, we can use the coefficients as the terms in a system involving $p_i(t)$ only. Since all the functions $\Phi_i$ are continuous in $t$ and due to the uniqueness of solutions to linear differential equations with continuous coefficients, it follows in particular that $p_i(0)=0$ implies $p_i(t)=0$ for all $t\in[0,T)$.
\end{proof}

We observe that Lemma~\ref{lem:zeromomenta} implies that for solutions to Hamilton's equations~\eqref{eq:hamx} and \eqref{eq:hamp} subject to the initial condition $p_3(0)=\dots=p_n(0)=0$, the dynamics of the first two landmarks is not affected by the remaining landmarks and is completely described by the system of equations
\begin{align}\label{2landmark1dsys}
\begin{aligned}
    \frac{\db x_1}{\db t}&=\kernl(0)p_1+\kernl(\|x_1-x_2\|)p_2, \qquad
    & \frac{\db p_1}{\db t}
    &=-\kernl'(\|x_1-x_2\|)\frac{x_1-x_2}{\|x_1-x_2\|}\langle p_1,p_2\rangle,\\
    \frac{\db x_2}{\db t}&=\kernl(\|x_1-x_2\|)p_1+\kernl(0)p_2,  \qquad  
    &\frac{\db p_2}{\db t}
    &=\kernl'(\|x_1-x_2\|)\frac{x_1-x_2}{\|x_1-x_2\|}\langle p_1,p_2\rangle.
\end{aligned}
\end{align}
We will perform a more detailed analysis of the above system in Section~\ref{sec:twolandmarksystem}. For the remainder of this section though, it suffices to consider initial conditions which further simply the above system. The behavior of the remaining landmarks $x_3,\dots,x_n$ is non-trivial but irrelevant for the subsequent analysis.

\begin{lemma}\label{lem:shooting}
    Let $(\lm,\mom)\colon [0,T)\to T^\star M$ for $T>0$ be a solution to the system consisting of Hamilton's equations~\eqref{eq:hamx} and \eqref{eq:hamp} whose initial value satisfies $p_1(0)=-p_2(0)$, $p_3(0)=\dots=p_n(0)=0$ and $x_1(0)=-x_2(0)$. For such a solution, we have
    \begin{displaymath}
        p_1(t)=-p_2(t)\text{ and }
        x_1(t)=-x_2(t)\text{ for all }t\in[0,T).
    \end{displaymath}
\end{lemma}
\begin{proof}
    We know from Lemma~\ref{lem:zeromomenta} that the initial condition $p_3(0)=\dots=p_n(0)=0$ gives rise to $p_3(t)=\dots=p_n(t)=0$ for all $t\in[0,T)$. This implies that the total momentum $P\colon[0,T)\to\R^d$ defined in~\eqref{eq:totalmom} is simply given by
    \begin{displaymath}
        P(t)=p_1(t)+p_2(t).
    \end{displaymath}
    Since we further have $P(0)=p_1(0)+p_2(0)=0$ by assumption, the conservation of total momentum provided by Lemma~\ref{lem:conservationofmomentum} yields $p_1(t)=-p_2(t)$ for all $t\in[0,T)$.

    As a result of~\eqref{2landmark1dsys} and $p_1(t)=-p_2(t)$, the second claim follows directly from
    \begin{displaymath}
        \frac{\db}{\db t}\left(x_1(t)+x_2(t)\right)
        =\kernl(0)p_1-\kernl(\|x_1-x_2\|)p_1+\kernl(\|x_1-x_2\|)p_1-\kernl(0)p_1=0
    \end{displaymath}
    as well as the assumption that $x_1(0)+x_2(0)=0$.
\end{proof}

We now establish Proposition~\ref{propn:geoincomplete} by essentially showing that if we single out two landmarks and shoot them together, then this gives rise to a geodesic which cannot be extended to all times provided that, for some $a>0$ and hence for all $a>0$,
\begin{displaymath}
    \int_0^a \frac{\db r}{\sqrt{\kernl(0)-\kernl(r)}}<\infty.
\end{displaymath}

\begin{proof}[Proof of Proposition~\ref{propn:geoincomplete}]
    We take a maximal solution $(\lm,\mom)\colon [0,T)\to T^\star M$ for $T>0$ to the system consisting of Hamilton's equations~\eqref{eq:hamx} and \eqref{eq:hamp} subject to the initial condition $p_1(0)=-p_2(0)$, $p_3(0)=\dots=p_n(0)=0$ and $x_1(0)=-x_2(0)$. Let $u\colon [0,T)\to\R^d$ and $q\colon [0,T)\to\R^d$ be defined by
    \begin{displaymath}
        u(t)=x_1(t)-x_2(t)\quad\text{and}\quad q(t)=p_1(t).
    \end{displaymath}
    According to Lemma~\ref{lem:shooting}, we additionally have that $q(t)=-p_2(t).$
    From Hamilton's equations~\eqref{eq:hamx} and \eqref{eq:hamp}, or alternatively directly from~\eqref{2landmark1dsys}, we obtain that $u$ and $q$ satisfy
    \begin{align}
        \frac{\db u}{\db t}\label{eq:reducedhamu}
        &=2\left(\kernl(0)-\kernl(\|u\|)\right)q,\\
        \frac{\db q}{\db t}\label{eq:reducedhamq}
        &=\kernl'(\|u\|)\frac{u}{\|u\|}\|q\|^2.
    \end{align}
    We observe that the system of equations~\eqref{eq:reducedhamu} and \eqref{eq:reducedhamq} is a system of equations for each component of the vector-valued functions $u$ and $q$ separately, with the same coefficients depending on $\|u\|$ and $\|q\|$ only. It follows that if $u$ and $q$ both have the same components initially zero, then those components will remain zero for all $t\in[0,T)$.

    In particular, if all but the first component of $u$ and $q$, respectively, are initially zero, then $u$ and $q$ will stay on this axis in $\R^d$ for all $t\in[0,T)$. In such a situation, we may also treat the functions $u$ and $q$ as real-valued, keeping track of their non-zero component only, and the equations~\eqref{eq:reducedhamu} and \eqref{eq:reducedhamq} as real-valued equations rather than vector-valued ones.

    Let us now assume that we indeed consider such a solution with $u(0)=a>0$ and $q(0)<0$. The equation~\eqref{eq:reducedhamu} then simplifies to the real-valued equation
    \begin{equation}\label{eq:ODE4u}
        \frac{\db u}{\db t} =2\left(\kernl(0)-\kernl(u)\right)q.
    \end{equation}
    We observe from~\eqref{eq:ODE4u} that, because $\kernl$ is strictly decreasing on $[0,\infty)$, the choice $q(0)<0$ ensures that $u$ is also decreasing.

    Since the Hamiltonian $\ham\colon T^\star M\to\R$ given by~\eqref{eq:hamiltonian} is time-independent, it is conserved along any solution to the system of Hamilton's equations~\eqref{eq:hamx} and \eqref{eq:hamp}. Therefore, we have, for all $t\in[0,T)$ and for some constant $E>0$ depending on the initial condition,
    \begin{displaymath}
        E=\left(\kernl(0)-\kernl(u(t))\right) q(t)^2.
    \end{displaymath}
    Keeping in mind $q(0)<0$ and using the above energy conservation law to eliminate $q$ from~\eqref{eq:ODE4u}, we obtain
    \begin{equation}\label{eq:separable}
        \frac{\db u}{\db t}=-2\sqrt{E}\sqrt{\kernl(0)-\kernl(u)}.
    \end{equation}
    The only possible equilibria of this autonomous differential equation are solutions to the equation $\kernl(0)-\kernl(u)=0$, that is, due to $\kernl$ being strictly decreasing, at $u=0$. The separable differential equation~\eqref{eq:separable} can then be solved to show that the function $u$ decreases from $u(0)=a>0$ to zero in time 
    \begin{displaymath}
        T=\int_0^a\frac{\db u}{2\sqrt{E}\sqrt{\kernl(0)-\kernl(u)}}.
    \end{displaymath}
    Thus, if the above integral is finite, which is exactly given by the condition~\eqref{eq:kernlintegrabl}, the chosen initial condition gives rise to a geodesic on $(M,g)$ for which the two landmarks $x_1$ and $x_2$ meet in finite time, that is, this geodesic cannot be extended to all times.
\end{proof}

\begin{remark}\label{rem:momentumblowup}
    If we choose a solution $(\lm,\mom)\colon [0,T)\to T^\star M$ to the system of Hamilton's equations~\eqref{eq:hamx} and \eqref{eq:hamp} with initial condition as in the proof of Proposition~\ref{propn:geoincomplete} then, due to the conservation of the Hamiltonian along the solution, we know that, for all $t\in[0,T)$ and for some constant $E>0$,
    \begin{displaymath}
        E=(\kernl(0)-\kernl(\|u(t)\|))\|q(t)\|^2.
    \end{displaymath}
    Hence, if we have a solution with $\lim_{t\uparrow T} \|u(t)\|=0$ then
    \begin{displaymath}
        \lim_{t\uparrow T}\|q(t)\|=\infty,
    \end{displaymath}
    that is, the linear momentum blows up as landmarks approach collision.
\end{remark}

The following example illustrates the construction used in the proof of Proposition~\ref{propn:geoincomplete} and, thereby, demonstrates for the special case of the Laplacian kernel the phenomenon of two landmarks moving along a geodesic and colliding in finite time. This is discussed in more detail by Camassa, Kuang and Lee in~\cite{camassa2016solitary}.

\begin{example}
    We consider $n=2$ landmarks in $\R$, that is, $d=1$, as well as the kernel corresponding to $\kernl\colon [0,\infty)\to\R$ given by $\kernl(r)=\exp(-r)$. Fix $b,T\in(0,\infty)$ and define $(\lm,\mom)\colon [0,T)\to T^\star M$ by, for $t\in[0,T)$,
    \begin{align*}
        x_1(t)&=\log\cosh\left(b(T-t)\right),
        & p_1(t) &=-\frac{b}{\tanh\left(b(T-t)\right)},\\
        x_2(t)&=-\log\cosh\left(b(T-t)\right),
        & p_2(t) &=\frac{b}{\tanh\left(b(T-t)\right)}.
    \end{align*}
    Observe that $x_1(t)=-x_2(t)>0$ and $p_1(t)=-p_2(t)<0$ for all $t\in[0,T)$. Moreover, one can verify that $(\lm,\mom)\colon [0,T)\to T^\star M$ is a solution to the system of Hamilton's equations~\eqref{2landmark1dsys}. This solution satisfies
    \begin{displaymath}
        \lim_{t\uparrow T} x_1(t) = 0 = \lim_{t\uparrow T} x_2(t),
    \end{displaymath}
    that is, the two landmarks move along a geodesic and collide in time $T<\infty$. In addition, we see
    \begin{displaymath}
        \lim_{t\uparrow T} p_1(t)=-\infty
        \quad\text{and}\quad
        \lim_{t\uparrow T} p_2(t)=\infty,
    \end{displaymath}
    which illustrates the observation made in Remark~\ref{rem:momentumblowup}.
\end{example}

\section{Geodesically complete landmark space}
\label{sec:geocomplete}

We first establish that, under the assumption that our integrability condition~\eqref{eq:kernlintegrabl} is not satisfied,
any curve which tends to the collision set $C=\R^{nd}\setminus M$ or to Euclidean infinity in finite time has infinite length, where those two cases are treated separately. We then obtain Proposition~\ref{propn:geocomplete} as a consequence of the Hopf--Rinow theorem.

For $i,j\in\{1,\dots,n\}$ fixed with $i\not=j$, let $f_{i}\colon \R^{nd}\to\R$ and $f_{ij}\colon \R^{nd}\to\R$ be the Euclidean distance functions defined by, for $\lm=(x_1,\dots,x_n)\in\R^{nd}$,
\begin{displaymath}
    f_{i}(\lm)=\|x_i\|
    \quad\text{and}\quad
    f_{ij}(\lm)=\|x_i-x_j\|.
\end{displaymath}

The first lemma shows that curves which tend to the collision set $C$ in finite time must have infinite length.

\begin{lemma}\label{lem:nocollision}
    Fix $i,j\in\{1,\dots,n\}$ with $i\not=j$ and suppose that, for all $a>0$,
    \begin{displaymath}
        \int_0^a\frac{\db r}{\sqrt{\kernl(0)-\kernl(r)}}=\infty.
    \end{displaymath}
    If $\gamma\colon [0,T)\to M$ is a piecewise differentiable curve in $M$ such that
    $\lim_{t\uparrow T} f_{ij}(\gamma(t))=0$, then $\gamma$ has infinite length with respect to the metric $g$ on $M$.
\end{lemma}
\begin{proof}
    Since $x_i\not=x_j$ for any $\lm\in M$, we have $f_{ij}(\gamma(t))>0$ for all $t\in[0,T)$.
    We then observe that the one-form $\db f_{ij}$ is given by
    \begin{displaymath}
        \left(\db f_{ij}\right)(\lm)
        =\frac{\langle x_i-x_j,\db x_i-\db x_j\rangle}{f_{ij}(\lm)}.
    \end{displaymath}
    In particular, all components of $\db f_{ij}$, except for
    \begin{displaymath}
        \left(\db f_{ij}\right)^i(\lm)
        =\frac{\left(x_i-x_j\right)^\top}{f_{ij}(\lm)}
        \quad\text{and}\quad
        \left(\db f_{ij}\right)^j(\lm)
        =\frac{\left(x_j-x_i\right)^\top}{f_{ij}(\lm)},
    \end{displaymath}
    are zero. This together with the cometric $g^{-1}$ of the metric $g$ being given by~\eqref{eq:defnginv} implies that
    \begin{displaymath}
        g^{-1}(\lm)(\db f_{ij},\db f_{ij})
        =\frac{1}{\left(f_{ij}(\lm)\right)^2}
        \left(\kernl(0)\|x_i-x_j\|^2-2\kernl(\|x_i-x_j\|)\|x_i-x_j\|^2+\kernl(0)\|x_j-x_i\|^2\right),
    \end{displaymath}
    which reduces to
    \begin{equation}\label{eq:ginvondf}
        g^{-1}(\lm)(\db f_{ij},\db f_{ij})
        =2\left(\kernl(0)-\kernl\left(f_{ij}(\lm)\right)\right).
    \end{equation}
    We now use the notation $\dot{\gamma}=\frac{\db\gamma}{\db t}$. Since we can write, for a suitable positive-definite symmetric matrix $G$ and with respect to the Euclidean inner product on $\R^{nd}$,
    \begin{displaymath}
        g\left(\dot{\gamma},\dot{\gamma}\right)
        =\left\langle\dot{\gamma}, G \dot{\gamma}\right\rangle
        \quad\text{as well as}\quad
        g^{-1}(\db f_{ij},\db f_{ij})
        =\left\langle\left(\db f_{ij}\right)^\top,G^{-1}\left(\db f_{ij}\right)^\top\right\rangle,
    \end{displaymath}
    the Cauchy--Schwarz inequality yields
    \begin{displaymath}
        g\left(\dot{\gamma},\dot{\gamma}\right)
        g^{-1}(\db f_{ij},\db f_{ij})
        =\left\|\sqrt{G}\dot{\gamma}\right\|^2
        \left\|\sqrt{G^{-1}}\left(\db f_{ij}\right)^\top\right\|^2\\
        \geq \left\langle\sqrt{G}\dot{\gamma}, \sqrt{G^{-1}}\left(\db f_{ij}\right)^\top\right\rangle^2
        =\left\langle\dot{\gamma},\left(\db f_{ij}\right)^\top\right\rangle^2.
    \end{displaymath}
    By~\eqref{eq:ginvondf} and due to the action of the one-form $\db f_{ij}$ on vectors, it follows that, for all $t\in[0,T)$,
    \begin{displaymath}
        \sqrt{g\left(\dot{\gamma}(t),\dot{\gamma}(t)\right)}
        \geq \frac{\left|\frac{\db}{\db t} \left(f_{ij}\circ \gamma\right)(t)\right|}{\sqrt{2\left(\kernl(0)-\kernl\left(f_{ij}\left(\gamma(t)\right)\right)\right)}}.
    \end{displaymath}
    Using the change of variables $r=f_{ij}(\gamma(t))$ and the assumption that $f_{ij}\circ \gamma\colon [0,T)\to \R$ goes from some positive value $a=f_{ij}(\gamma(t_0))$ for $t_0\in[0,T)$ to $\lim_{t\uparrow T} f_{ij}(\gamma(t))=0$, we obtain
    \begin{displaymath}
        \int_0^T \sqrt{g\left(\dot{\gamma}(t),\dot{\gamma}(t)\right)}\dd t
        \geq \int_0^T \frac{\left|\frac{\db}{\db t} \left(f_{ij}\circ \gamma\right)(t)\right|}{\sqrt{2\left(\kernl(0)-\kernl\left(f_{ij}\left(\gamma(t)\right)\right)\right)}}\dd t
        \geq \int_0^a \frac{\db r}{{\sqrt{2\left(\kernl(0)-\kernl(r)\right)}}}.
    \end{displaymath}
    Thus, as the latter improper integral does not converge by assumption, the curve $\gamma\colon[0,T)\to M$ has infinite length.
\end{proof}

The next lemma shows that also any curve which tends to Euclidean infinity in finite time must have infinite length. We highlight that this result is true regardless of the choice of smooth kernel $\kernl\colon[0,\infty)\to\R$ and, in particular, that it is independent of our integrability condition.

\begin{lemma}\label{lem:norunningaway}
    Fix $i\in\{1,\dots,n\}$.
    If $\gamma\colon[0,T)\to M$ is a piecewise differentiable curve in $M$ such that $\lim_{t\uparrow T} f_i(\gamma(t))=\infty$, then $\gamma$ has infinite length with respect to the metric $g$ on $M.$
\end{lemma}
\begin{proof}
    Away from $\lm\in M$ with $x_i=0$, we have
    \begin{displaymath}
        (\db f_i)(\lm)=\frac{\langle x_i,\db x_i\rangle}{f_i(\lm)},
    \end{displaymath}
    which gives rise to
    \begin{displaymath}
        g^{-1}(\lm)(\db f_i,\db f_i)
        =\frac{1}{\left(f_i(\lm)\right)^2}\left(\kernl(0)\|x_i\|^2\right)
        =\kernl(0).
    \end{displaymath}
    By proceeding as in the proof of Lemma~\ref{lem:nocollision}, the Cauchy--Schwarz inequality further shows that
    \begin{displaymath}
        g\left(\dot{\gamma},\dot{\gamma}\right)
        g^{-1}(\db f_{i},\db f_{i})
        \geq\left\langle\dot{\gamma},\left(\db f_{i}\right)^\top\right\rangle^2.
    \end{displaymath}
    Since $f_i\circ\gamma\colon[0,T)\to\R$ has to eventually go from some positive yet finite value $a=f_i(\gamma(t_0))$ for $t_0\in[0,T)$ to $\lim_{t\uparrow T} f_i(\gamma(t))=\infty$ without visiting zero, we then conclude that
    \begin{displaymath}
        \int_0^T \sqrt{g\left(\dot{\gamma}(t),\dot{\gamma}(t)\right)}\dd t
        \geq \int_0^T \frac{\left|\frac{\db}{\db t} \left(f_{i}\circ \gamma\right)(t)\right|}{\sqrt{\kernl(0)}}\dd t
        \geq \int_a^\infty \frac{\db r}{{\sqrt{\kernl(0)}}}=\infty,
    \end{displaymath}
    and the claimed result follows.
\end{proof}

\begin{proof}[Proof of Proposition~\ref{propn:geocomplete}]
    Since the scalar function $\kernl\colon[0,\infty)\to\R$ that defines the cometric kernel is assumed to be smooth on $(0,\infty)$, the Hopf--Rinow theorem implies that the Riemannian manifold $(M,g)$ is geodesically complete if and only if it is metrically complete. However, as a consequence of~\cite[Lemma~6]{discretesobolev}, metric completeness follows after showing that every curve which tends to the collision set $C$ or to Euclidean infinity in finite time must have infinite length, which is ensured by Lemma~\ref{lem:nocollision} and Lemma~\ref{lem:norunningaway}.
\end{proof}

\section{Geodesically complete and stochastically incomplete manifold}\label{sec:geostodiff}

We proceed by proving Theorem~\ref{thm:geostodiff} which provides a family of landmark spaces of exactly two landmarks that are geodesically complete but stochastically incomplete, where we use Theorem~\ref{thm:main} for the characterization of geodesic completeness and the approach developed in~\cite{LongtimeExistenceBrownian2024} to characterize stochastic incompleteness.

The landmark space of exactly two landmarks in any $\R^d$ is stochastically incomplete if the distance process $(r_t)_{t\in[0,\zeta)}$ between the two landmarks in Riemannian landmark Brownian motion has positive probability of hitting zero in finite time, which corresponds to the two landmarks colliding in finite time. As established in~\cite{LongtimeExistenceBrownian2024}, the stochastic process $(r_t)_{t\in[0,\zeta)}$ is the unique strong solution to the It\^o stochastic differential equation
\begin{equation}\label{eq:SDE}
  \db r_t =\sigma(r_t)\dd B_t+b(r_t)\dd t,
\end{equation}
for $(B_t)_{t\geq 0}$ a one-dimensional standard Brownian motion, $\sigma\colon [0,\infty)\to\R$ as well as $b\colon[0,\infty)\to\R$ given by
\begin{displaymath}
    \sigma(r)=\sqrt{2(\kernl(0)-\kernl(r))}
    \quad\text{and}\quad
    b(r)=\frac{((d-1)\kernl(r)-\kernl(0))\kernl'(r)}{\kernl(0)+\kernl(r)},
\end{displaymath}
and subject to $r_0>0$. 
Its behavior near zero can be studied, for example, by following the classification of singular points obtained by Cherny and Engelbert in~\cite{SSDE}. We specifically make use of the result~\cite[Theorem~2.13]{SSDE} stated below, where $\rho\colon (0,a]\to\R$ and $s\colon (0,a]\to\R$ are defined by
\begin{displaymath}
    \rho(r)=\exp\left(\int_r^a\frac{2b(y)}{\left(\sigma(y)\right)^2}\dd y\right)
    \quad\text{and}\quad
    s(r)=\int_0^r\rho(y)\dd y,
\end{displaymath}
respectively.
\begin{theorem}[Cherny--Engelbert~\cite{SSDE}]\label{thm:CE213}
    Suppose that, for some $a>0$,
    \begin{displaymath}
        \int_0^a\rho(r)\dd r < \infty,\quad
        \int_0^a\frac{1+|b(r)|}{\rho(r)\left(\sigma(r)\right)^2}\dd r=\infty
        \quad\text{and}\quad
        \int_0^a\frac{1+|b(r)|}{\rho(r)\left(\sigma(r)\right)^2}s(r)\dd r<\infty,
    \end{displaymath}
    and set $T_{0,a}=\inf\{t\geq 0:r_t=0\text{ or }r_t=a\}$.
    Then, if $r_0\in [0,a]$, there exists a unique solution $(r_t)_{t\in[0,T_{0,a}]}$ to~\eqref{eq:SDE}. Moreover, we have $\mathbb{E}[T_{0,a}]<\infty$ and $\mathbb{P}(r_{T_{0,a}}=0)>0$, that is, with positive probability zero is hit in finite time.
\end{theorem}

\begin{proof}[Proof of Theorem~\ref{thm:geostodiff}]
    Using the change of variables $u=1-\log(r)$, we first compute, for $r_0\in(0,1)$ and $\beta\in\R$,
    \begin{displaymath}
        \int_{r_0}^1\frac{\dd r}{r\left(1-\log(r)\right)^\beta}
        =\int_1^{1-\log(r_0)}u^{-\beta} \dd u
        =
        \begin{cases}
            \log(1-\log(r_0))& \text{if }\beta=1,\\
            \frac{1}{1-\beta}\left(\left(1-\log(r_0)\right)^{1-\beta}-1\right)& \text{if }\beta\not= 1,
        \end{cases}
    \end{displaymath}
    which implies as $r_0\to 0$ that
    \begin{displaymath}
        \int_0^1\frac{\dd r}{r\left(1-\log(r)\right)^\beta}
        =
        \begin{cases}
            \infty& \text{if }\beta\leq 1,\\
            \frac{1}{\beta -1}& \text{if }\beta> 1.
        \end{cases}
    \end{displaymath}
    Since $\kernl\colon[0,\infty)\to\R$ is chosen such that, for some constant $c\in(1,2]$ and for $r\in[0,\frac{1}{2}]$,
    \begin{displaymath}
        \kernl(r)=1-r^2(1-\log(r))^{c},
    \end{displaymath}
    it particularly follows that, for $a>0$,
    \begin{equation}\label{eq:integrals}
        \int_0^a\frac{\dd r}{\sqrt{\kernl(0)-\kernl(r)}}=\infty
        \quad\text{and}\quad
        \int_0^a\frac{r}{\kernl(0)-\kernl(r)}\dd r <\infty.
    \end{equation}
    
    Applying Theorem~\ref{thm:main}, we deduce from~\eqref{eq:integrals} that the Riemannian landmark manifold $(M,g)$ associated with this choice of $\kernl$ is geodesically complete for any number of landmarks in any $\R^d$.
    
    It remains to study its stochastic completeness or stochastic incompleteness, for which we have to restrict our attention to exactly two landmarks to be able to apply the approach developed in~\cite{LongtimeExistenceBrownian2024}. As derived in~\cite{LongtimeExistenceBrownian2024}, we have
    \begin{displaymath}
        \rho(r)=
        \left(\frac{\kernl(0)-\kernl(a)}{\kernl(0)-\kernl(r)}\right)^{1-d/2}\left(\frac{\kernl(0)+\kernl(a)}{\kernl(0)+\kernl(r)}\right)^{-d/2},
    \end{displaymath}
    which shows that $\int_0^a\rho(r)\dd r<\infty$ for $d=2$.
    
    In the following, we write $f(r)\sim g(r)$ as $r\downarrow 0$ if there exists a non-zero constant $C\in\R$ such that, as $r\downarrow 0$,
    \begin{displaymath}
        f(r)=g(r)(C+o(1))\;.
    \end{displaymath}
    We then have, as $r\downarrow 0$ and using $d=2$,
    \begin{displaymath}
        \frac{1+|b(r)|}{\rho(r)\left(\sigma(r)\right)^2}\sim
        \frac{1}{\kernl(0)-\kernl(r)},
    \end{displaymath}
    so that we can use $\left(\kernl(0)-\kernl(r)\right)^{-1}\geq\left(\kernl(0)-\kernl(r)\right)^{-1/2}$ to obtain
    \begin{displaymath}
        \int_0^a\frac{1+|b(r)|}{\rho(r)\left(\sigma(r)\right)^2}\dd r=\infty.
    \end{displaymath}
    We further observe that, for $d=2$,
    \begin{displaymath}
        \frac{1+|b(r)|}{\rho(r)\left(\sigma(r)\right)^2}s(r)\sim
        \frac{r}{\kernl(0)-\kernl(r)},
    \end{displaymath}
    which with our particular choice of scalar function $\kernl$, as a result of~\eqref{eq:integrals}, gives rise to
    \begin{displaymath}
        \int_0^a\frac{1+|b(r)|}{\rho(r)\left(\sigma(r)\right)^2}s(r)\dd r<\infty.
    \end{displaymath}
    According to Theorem~\ref{thm:CE213}, that is~\cite[Theorem~2.13]{SSDE}, it finally follows that the stochastic process $(r_t)_{t\in[0,T_{0,a}]}$ hits zero in finite time with positive probability. Thus, the landmark space $(M,g)$ of exactly two landmarks associated with $\kernl$ is stochastically incomplete because the two landmarks collide in finite time with positive probability.
\end{proof}

\section{Two-landmark system with non-zero angular momentum}
\label{sec:twolandmarksystem}

In Section~\ref{sec:geoincomplete}, we demonstrate that under the integrability condition~\eqref{eq:kernlintegrabl} particular choices of initial conditions give rise to geodesics that cannot be extended to all times. Notably, the total angular momentum of all these systems is zero. By extending the analysis in the Hamiltonian approach, we establish Theorem~\ref{thm:twolandmarks}, rephrased more precisely as Theorem~\ref{thm:twolandmarksrig} below, that in the two-landmark system non-zero total angular momentum prevents finite-time breakdown for all choices of kernels.

As already observed by McLachlan and Marsland in~\cite{MM}, the two-landmark system is completely integrable. However, the work~\cite{MM} crucially relies on the assumption that the function defined on $\R$ by $x\mapsto \kernl(|x|)$ is twice continuously differentiable at zero, whereas we more generally consider kernels that are continuous at the collision set but not necessarily any better than that. The landmark manifolds with exactly two landmarks have also been studied by Micheli, Michor and Mumford in~\cite{MMM}, but again not for cometric kernels which admit collision of landmarks. Numerical simulations for the two-landmark system have been performed by~Chertock, Du Toit and Marsden in~\cite{chertock2012integration}, and exact solutions for the special case of the Laplacian kernel are derived by Camassa, Kuang and Lee in~\cite{camassa2016solitary}, where collisions in finite time are demonstrated by means of these solutions.

As discussed in Section~\ref{sec:geoincomplete}, for landmark space with exactly two landmarks in any $\R^d$, that is, for $n=2$ and any positive number $d\geq 1$, Hamilton's equations~\eqref{eq:hamx} and \eqref{eq:hamp} reduce to the system~\eqref{2landmark1dsys} of equations.
In terms of $r=\|x_1-x_2\|$ and the center-of-mass coordinates $(u,v,P,Q)$ for $T^\star M$ used in~\cite{MM,MMM} and defined by
\begin{displaymath}
    u=x_1-x_2,\quad
    v=\frac{x_1+x_2}{2},\quad
    P=p_1+p_2\quad\text{and}\quad
    Q=\frac{p_1-p_2}{2},
\end{displaymath}
Hamilton's equations for the two-landmark system write as
\begin{align*}
    \frac{\db u}{\db t}&=2\left(\kernl(0)-\kernl(r)\right)Q, 
    & \frac{\db P}{\db t}
    &=0,\\[0.4em]
    \frac{\db v}{\db t}&=\frac{1}{2}\left(\kernl(0)+\kernl(r)\right)P,    
    & \frac{\db Q}{\db t}
    &=\frac{\kernl'(r)u}{r}\left(\|Q\|^2-\frac{1}{4}\|P\|^2\right).
\end{align*}
Since the total momentum $P$ is conserved along any solution to the system of Hamilton's equations, as has already been observed in Lemma~\ref{lem:conservationofmomentum}, there exists some constant $c\geq 0$ such that $c=\frac{1}{4}\|P\|^2$. Consequently, the equations for $u$ and $Q$ are independent of the equation for $v$, and we then focus on studying the system of equations
\begin{align}
    \label{eq:4u}
    \frac{\db u}{\db t} &=2\left(\kernl(0)-\kernl(r)\right)Q,\\[0.4em]
    \label{eq:4Q}
    \frac{\db Q}{\db t} &=\frac{\kernl'(r)u}{r}\left(\|Q\|^2-c\right).
\end{align}
This system has a further two conservation laws. Firstly, the conservation of the Hamiltonian
\begin{displaymath}
    \ham
    =\frac{1}{2}\kernl(0)\left(\|p_1\|^2+\|p_2\|^2\right)+\kernl(r)\langle p_1,p_2\rangle
    =2c\kernl(0)+\left(\kernl(0)-\kernl(r)\right)\left(\|Q\|^2-c\right)
\end{displaymath}
implies that 
\begin{displaymath}
D=\left(\kernl(0)-\kernl(r)\right)\left(\|Q\|^2-c\right)
\end{displaymath}
is constant. Alternatively, this can be derived directly from the equations~\eqref{eq:4u} and \eqref{eq:4Q}.
Secondly, the angular momentum $u\wedge Q$ of the system described by~\eqref{eq:4u} and \eqref{eq:4Q} is also conserved, which is a consequence of Lemma~\ref{lem:conservationofangmom} and the fact that, for the two-landmark system, additionally the quantity $x_1\wedge p_2+x_2\wedge p_1$ is conserved. More conveniently, making use of~\eqref{eq:4u} as well as \eqref{eq:4Q} and exploiting the antisymmetry of the wedge product, we obtain
\begin{displaymath}
    \frac{\db}{\db t}\left(u \wedge Q\right)
    =\frac{\db u}{\db t}\wedge Q +u\wedge\frac{\db Q}{\db t}
    =2\left(\kernl(0)-\kernl(r)\right)Q\wedge Q
    +u\wedge \left(\frac{\kernl'(r)u}{r}\left(\|Q\|^2-\frac{1}{4}\|P\|^2\right)\right)=0,
\end{displaymath}
which immediately proves the conservation of $u\wedge Q$.

With respect to the Euclidean norm on $\R^{\frac{d(d-1)}{2}}$, we set $\omega=\|u\wedge Q\|$. As a consequence of the following lemma, we have
\begin{equation}\label{eq:identityforomega}
    \omega^2+ \langle u,Q\rangle^2=r^2 \|Q\|^2.
\end{equation}

\begin{lemma}
    For $y,z\in\R^d$, we have
    \begin{displaymath}
        \|y\wedge z\|^2+\langle y,z\rangle^2 =\|y\|^2\|z\|^2.
    \end{displaymath}
\end{lemma}
\begin{proof}
    We obtain
    \begin{align*}
        \|y\wedge z\|^2
        &=\sum_{1\leq k<\ell\leq d}\left(y_kz_\ell-y_\ell z_k\right)^2\\
        &=\sum_{1\leq k<\ell\leq d} y_k^2z_\ell^2 + \sum_{1\leq k<\ell\leq d} y_\ell^2z_k^2
         -2\sum_{1\leq k<\ell\leq d}y_ky_\ell z_kz_\ell\\
        &=\sum_{k,\ell=1}^d \left(y_k^2z_\ell^2-y_ky_\ell z_kz_\ell\right)
        =\|y\|^2\|z\|^2-\langle y,z\rangle^2,
    \end{align*}
    which establishes the claimed result.
\end{proof}

We can now prove Theorem~\ref{thm:twolandmarks}, rephrased more precisely as follows.

\begin{theorem}\label{thm:twolandmarksrig}
    Let $(M,g)$ be the landmark space of exactly two landmarks in $\R^d$ and suppose that $\kernl\colon[0,\infty)\to\R$ is smooth on $(0,\infty)$. Consider a maximal solution $(u,Q)\colon[0,T)\to T^\star \R^d$ to the system of equations~\eqref{eq:4u} and \eqref{eq:4Q} for some $T>0$ and with $\omega=\|u\wedge Q\|\not=0$. Then we are guaranteed that $T=\infty$, that is, finite-time breakdown cannot occur, regardless of the kernel $\kernl$.
\end{theorem}
\begin{proof}
    As a result of Lemma~\ref{lem:norunningaway}, it suffices to rule out the breakdown $\lim_{t\uparrow T}r(t)=0$ for $T<\infty$, which corresponds to collision of the two landmarks in finite time.

    We recall from the preceding discussion that $D=\left(\kernl(0)-\kernl(r(t))\right)\left(\|Q(t)\|^2-c\right)$ is constant for $t\in[0,T)$. In particular, due to $\kernl$ being a strictly decreasing function on $[0,\infty)$, the sign of $\|Q(t)\|^2-c=-\langle p_1(t),p_2(t)\rangle$ is preserved. We further observe that we cannot have $Q(t)=0$ for any $t\in[0,T)$ as this would contradict the assumption $\omega>0$. We now analyze the two cases $D\leq 0$ and $D>0$ separately. Note that for the corresponding solution on $T^\star M$, the assumption $D\leq 0$ amounts to supposing $\langle p_1(0),p_2(0)\rangle\geq 0$, whereas $D>0$ corresponds to $\langle p_1(0),p_2(0)\rangle<0$.

    \subparagraph{Case 1: $D\leq 0$.} In this case, we have $\|Q(t)\|^2\leq c$ for all $t\in [0,T)$. From the identity~\eqref{eq:identityforomega}, we further deduce that, for all $t\in[0,T)$,
    \begin{displaymath}
        r(t)^2\|Q(t)\|^2\geq \omega^2.
    \end{displaymath}
    Using $c\geq\|Q(t)\|^2> 0$, we obtain, for all $t\in[0,T)$,
    \begin{displaymath}
        r(t)^2\geq \frac{\omega^2}{c}>0.
    \end{displaymath}
    In particular, this excludes the breakdown $\lim_{t\uparrow T}r(t)=0$.

    \subparagraph{Case 2: $D> 0$.} We observe from~\eqref{eq:identityforomega} that in this case, for all $t\in[0,T)$,
    \begin{equation}\label{eq:inequality4r}
        r(t)^2\left(\frac{D}{\kernl(0)-\kernl(r(t))}+c\right)\geq \omega^2.
    \end{equation}
    Suppose we have $\lim_{t\uparrow T} r(t)=0$ for some $T>0$, then taking the limit $t\uparrow T$ in~\eqref{eq:inequality4r} implies that
    \begin{displaymath}
        \lim_{r\downarrow 0}\frac{D r^2}{\kernl(0)-\kernl(r)}\geq \omega^2.
    \end{displaymath}
    In particular, we can choose some $a>0$ such that, for all $r\in(0,a]$,
    \begin{displaymath}
        \frac{Dr^2}{\kernl(0)-\kernl(r)} \ge \frac{\omega^2}{2}.
    \end{displaymath}
    Since $D>0$ and $\omega>0$, it follows that
    \begin{displaymath}
        \int_0^a\frac{\db r}{\sqrt{\kernl(0)-\kernl(r)}}
        \geq \frac{\omega}{\sqrt{2D}}\int_0^a\frac{\db r}{r}=\infty,
    \end{displaymath}
    that is, the kernel $\kernl$ satisfies the conditions of Proposition~\ref{propn:geocomplete}, and we have global existence for this geodesic and, in fact, all geodesics.
\end{proof}

\bibliographystyle{plain}
\bibliography{references}

\end{document}